\documentclass[12pt]{amsart}
\usepackage{amsfonts, amsmath, amsthm}
\usepackage{graphicx}
\usepackage{fourier}
\usepackage{geometry}
\usepackage{color}
\usepackage{psfrag}
\usepackage{booktabs}
\usepackage{multirow}

\newtheorem{pro}{Proposition}[section]
\newtheorem{thm}[pro]{Theorem}
\newtheorem{lem}[pro]{Lemma}

\newtheorem{cor}[pro]{Corollary}

\theoremstyle{definition}

\newtheorem{notation}[pro]{Notation}

\theoremstyle{remark}

\def\kt{\ensuremath{K^{\rm tw}}}
\def\ktp{\ensuremath{K^{\rm tw}_+}}
\def\ktm{\ensuremath{K^{\rm tw}_-}}
\def\kn{\ensuremath{K^n}}
\def\krsn{\ensuremath{K_{r,s}^n}}
\def\k0{\ensuremath{K^0}}
\def\gr{\ensuremath{{\rm gr}_t}}
\newcommand{\X}{\text{\sffamily X}}

\title[The spectrum of the growth rate is infinite]{The spectrum of the growth rate of the tunnel number is infinite}

\date{\today}
\address{Department of Mathematics,  University of Miami, Coral Gables, FL, 33146}
\email{k.baker@math.miami.edu}
\address{Department of Mathematics, Nara Women's University
Kitauoya Nishimachi, Nara 630-8506, Japan}
\email{tsuyoshi@cc.nara-wu.ac.jp}
\address{Department of Mathematical Sciences, University of
Arkansas, Fayetteville, AR 72701}
\email{yoav@uark.edu}
\author{Kenneth L Baker}
\author{Tsuyoshi Kobayashi}
\author{Yo'av Rieck}
\thanks{KB and YR would like to thank Nara Women's University for their hospitality during the development of this article.  
This work was partially supported by grants from the Simons Foundation (\#209184 to Kenneth L. Baker and \#283495 to Yo'av Rieck).  
TK was supported by Grant-in-Aid for scientific research, JSPS grant number 00186751.}

\begin{document}

\subjclass[MCS2010]{57M99, 57M25}%

\date{\today}%
\begin{abstract}
For any $\epsilon > 0$ we construct a hyperbolic knots $K \subset S^{3}$
for which $1 - \epsilon < \gr(K) < 1$.  This shows that the spectrum
of the growth rate of the tunnel number is infinite.
\end{abstract}
\maketitle

\section{Introduction}
\label{sec:introduction}

By {\it manifold} we mean a compact connected orientable 3-manifold 
(the terms and notation used in this paper are introduced in the next section).
Let $K$ be a knot in a closed manifold $M$.  We denote the connected sum
of $n$ copies of $K$ by $nK$.  We use $E(\cdot)$ for knot or link exterior and $g(\cdot)$ for Heegaard genus.
In~\cite{growthrate1} Kobayashi and Rieck defined the growth rate of the tunnel 
number of a knot $K$ to be:
\[\gr(K) = \limsup_{n \to \infty} \frac{g(E(nK)) - ng(E(K)) + n-1}{n-1}\]
$K$ is called {\it admissible} if $g(E(K)) > g(M)$; note that every knot in the 3-sphere $S^{3}$
is admissible.  In~\cite{growthrate1} it was shown that for an admissible knot $K$, 
$\gr(K) < 1$, and $\gr(K) = 1$ otherwise.  The concept of growth rate proved to be useful in~\cite{MC1}
and~\cite{MC2} where it was used to construct counterexamples to Morimoto's Conjecture~\cite{morimoto};
however, calculating the growth rate proved to be very elusive.  The only numbers known to be
in the spectrum of the growth rate are:
\begin{description}
\item[0] torus knots and 2-bridge knots have growth rate 0~\cite{growthrate1}.
\item[1] as mentioned above, inadmissible knots have growth rate 1~\cite{growthrate1}.
\item[-1/2] Kobayashi and Saito constructed knots with growth rate $-1/2$~\cite{kobayashisaito}.
\item[1/2]  The knots constructed by Morimoto
Sakuma and Yokota in~\cite{MSY} have growth rate $1/2$~\cite{growthrate2}.
\end{description}

The goal of this paper is to construct knots whose growth rate can be estimated.
We prove:

\begin{thm}
\label{thm:main}
For every $\epsilon>0$ there exists a hyperbolic knot $K \subset S^3$
so that
\[1 - \epsilon < \gr(K) < 1\]
\end{thm}
As an immediate corollary we get:
\begin{cor}
The spectrum of the growth rate is infinite.
\end{cor}
To estimate the growth rate we apply~\cite{growthrate2} which we now explain;
we will only use the results of~\cite{growthrate2}
when $K$ is a knot in $S^3$ and $g(E(K)) = 2$.
Let $b_0(K)$ be the bridge index of $K$ and $b_1(K)$ be the torus bridge index of $K$;
for a detailed discussion see~\cite{growthrate2}.\footnote{In~\cite{growthrate2}
$b_0(K)$ is denoted by $b_2^*(K)$ and $b_1(K)$ is denoted by $b_1^*(K)$
but this notation is not needed in our restricted setting.}
It is easy to see that $b_0(K) > b_1(K)$ when $K$ is not the unknot.  
The main result of~\cite{growthrate2} is that if $K$ is m-small 
(for the definitions of m-small see the next section) and $g(E(K)) = 2$ then 
\[\gr(K) = \min \left\{1- \frac{1}{b_1(K)} , 1- \frac{2}{b_0(K)}\right\}\]
We prove Theorem~\ref{thm:main} by constructing m-small knots 
and bounding their bridge indices below.  For smallness (which implies
m-smallness for knots in $S^{3}$) we use Baker~\cite{baker}
and for the lower bound on the bridge indices
we apply Baker--Gordon--Luecke~\cite{BGL}.  We prove the following:
\begin{thm}
\label{thm:knots}
For $n=1,2,\dots$, there exists a hyperbolic knot $\kn \subset S^3$ 
so that the following three conditions hold:
\begin{enumerate}
\item $g(E(\kn)) = 2$.
\item $\kn$ is small (and hence m-small).
\item $\lim_{n \to \infty} b_1(\kn) = \infty$.
\end{enumerate}
\end{thm}
We now show how Theorem~\ref{thm:main} follows from Theorem~\ref{thm:knots}.
As \kn\ are m-small we may apply~\cite{growthrate2}; 
since $b_0(\kn) > b_1(\kn)$ and $\lim_{n \to \infty} b_1(\kn) = \infty$ we get: 
\[\lim_{n \to \infty} \gr(\kn) = \lim_{n \to \infty} \left(\max\left\{1 - \frac{1}{b_1(\kn)} , 1 - \frac{2}{b_0(\kn)}\right\} \right) = 1\]
On the other hand, since any knot in $S^3$ is admissible, by~\cite{growthrate1}
$\gr(\kn) < 1$.  This shows that Theorem~\ref{thm:main} follows from Theorem~\ref{thm:knots}.
The remainder of this paper is devoted to the proof of Theorem~\ref{thm:knots}.  

The structure of this paper is as follows:
\begin{enumerate}
\item[\S\ref{sec:prelims}] We describe some background material and the notation that we follow.
\item[\S\ref{sec:construction}] We construct the knots $\krsn$.  The construction depends
on three parameters ($r$,$s$, and $n$).  Since we usually think of $r$ and $s$ as fixed 
we often suppress $r$ and $s$ from the notation and denote the knots by \kn.
In Lemma~\ref{lem:GenusTwo} we show that $g(E(\kn)) \leq 2$.
\item[\S\ref{sec:small}] We apply Baker's~\cite{baker} and conclude 
that for fixed $r,s$ and sufficiently large $n$, $\krsn$ is small.  See Proposition~\ref{pro:small}.
\item[\S\ref{sec:bridge}] We apply Baker--Gordon--Luecke~\cite{BGL}
and conclude that for fixed $r,s$, $\lim_{n \to \infty} b_1(\kn) = \infty$.  See Proposition~\ref{pro:bridge}.
\item[\S\ref{sec:hyperbolicity}] We prove that for sufficiently large $r$, $s$, and $n$, \krsn\
is hyperbolic.  See Proposition~\ref{pro:hyperbolicity}.
\end{enumerate}
The combination of Lemma~\ref{lem:GenusTwo}, Proposition~\ref{pro:small}, Proposition~\ref{pro:bridge},
and Proposition~\ref{pro:hyperbolicity} proves Theorem~\ref{thm:knots}.

\section{Preliminaries}
\label{sec:prelims}

By {\it manifold} we mean compact orientable connected 3-dimensional manifold.
By {\it surface} we mean a compact orientable connected 2-dimensional manifold.
All surfaces are assumed to be properly embedded.  A surface embedded in a manifold 
is called {\it essential} if it is incompressible, boundary incompressible, not boundary 
parallel, and not a ball bounding sphere.  A manifold is called {\it small}
if it admits no closed essential surfaces.  We use $N(\cdot)$ for an open normal
neighborhood, $\partial$ for boundary, and $\mbox{cl}$ for closure.

By an (unoriented) {\it knot} $K$ in a 3-manifold $M$ we mean 
a smooth embedding of the circle into $M$, considered up to isotopy and reversal
of orientation of the circle.  
The exterior of $K$, denoted $E(K)$, is $M \setminus N(K)$.  
There is a unique isotopy class of simple closed curves on $\partial N(K)$,
called the {\it meridian} of $K$, that bounds a disk in the solid torus $\mbox{cl}(M \setminus E(K))$;
any simple closed curve in that isotopy class is called a {\it meridian}.
A knot $K \subset M$ is called {\it small} if $E(K)$ is small.  A knot is called
{\it m-small} (which stands for {\it meridionally small}) if $E(K)$ does not admit an essential 
surface with non-empty boundary, whose boundary consist of meridians. 
By~\cite[Theorem~2.0.3]{cgls} we have:
\begin{lem}
\label{lem:SamllIsMSmall}
Any small knot in $S^{3}$ is m-small.
\end{lem}

\section{The knots}
\label{sec:construction}

The construction given in this section is based on Baker's~\cite{baker}
and we refer the reader to that paper for a detailed discussion.

\begin{notation}
The following notation is fixed for the remainder of the paper:
let $3_1$ denote the left handed trefoil.  It is well known that its exterior, $E(3_1)$,
is fibered over the circle with fiber a once punctured torus.
Let $T$ be one such fiber (so $T$ is the intersection of 
a genus one Seifert surface for $3_1$ with $E(3_1)$).
Let $a$ and $b$ be the oriented
simple closed curves on $T$ given by Figure~\ref{fig:FiberT}.
\begin{figure}[h!]
\psfrag{a}[l][l]{$a$}
\psfrag{b}[l][l]{$b$}
\includegraphics[height=2in]{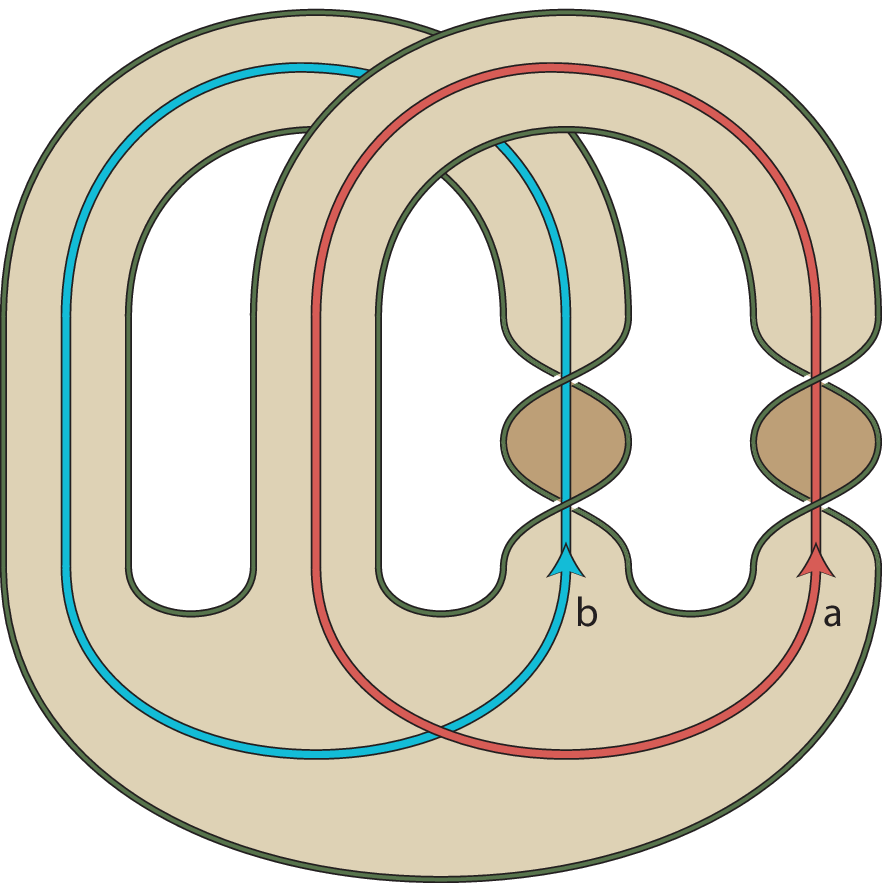} \qquad \qquad \includegraphics[height=2in]{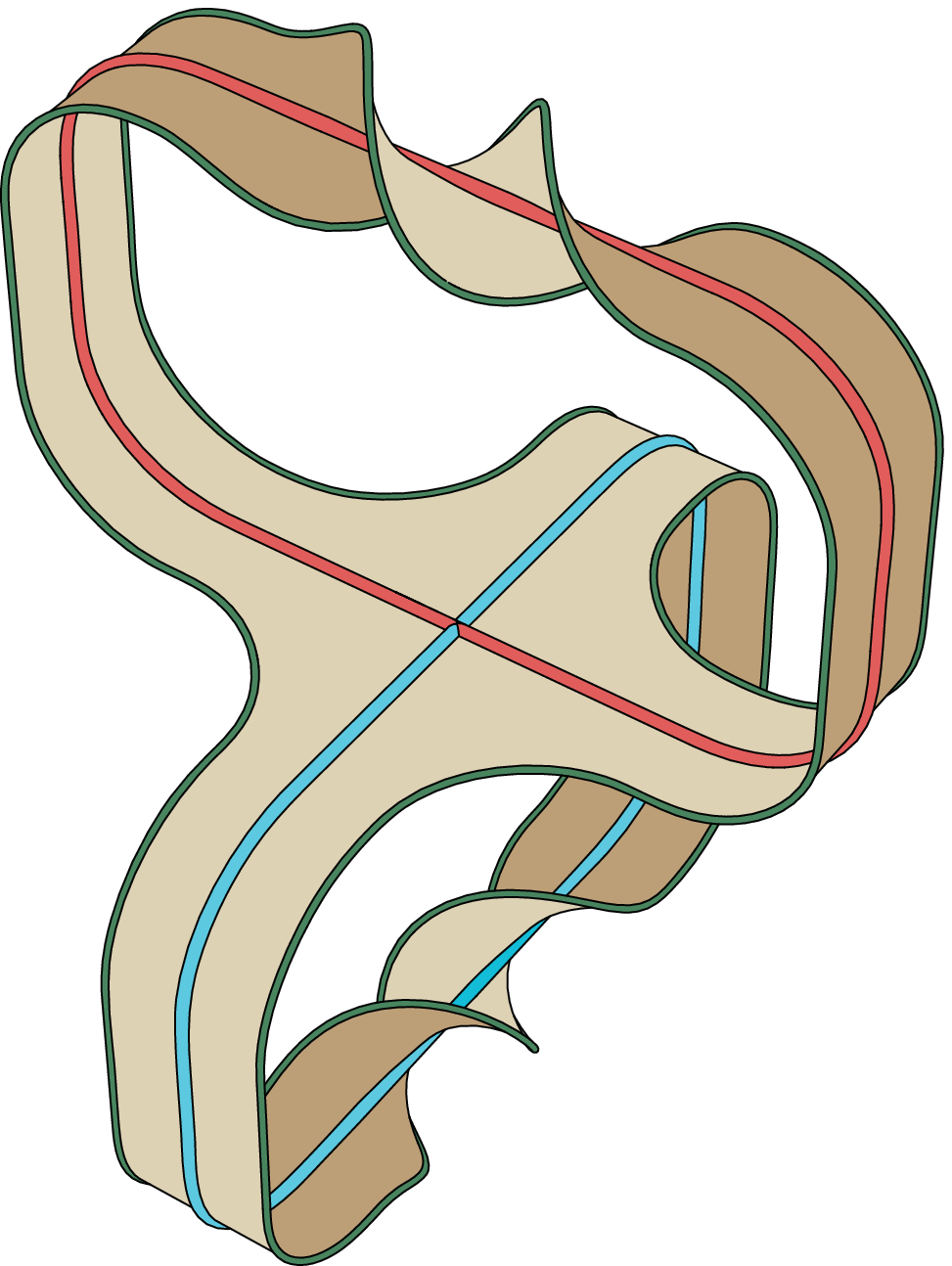}
\caption{The fiber $T$ with curves $a$ and $b$ that give a basis for homology.  The fiber is presented again 
on the right for comparison with the link $\mathcal{L}_7$ in Figure~\ref{fig:surgerydiagram}.}
\label{fig:FiberT}
\end{figure}
Any simple closed curve $K$ on $T$ is determined by its class in $H_1(T;\mathbb Z)$.  
Since $a$ intersects $b$ transversally exactly once, the classes $[a]$ and $[b]$ form a 
basis for $H_1(T;\mathbb Z)$.  Thus the homology class of $K$ 
can be written as $m[a]+n[b]$, for some relatively prime
integers $m$, $n$.  The coefficients $m$ and $n$ are almost unique: the only
other possibility is $-m$ and $-n$ (recall that we are considering unoriented knots).
We emphasize that we are considering the isotopy class of $K$
in $T$; it is distinctly possible for non isotopic curves on $T$ to be isotopic in $E(3_1)$.  As is customary
we identify the isotopy classes of $K$ (as an unoriented curve) with $m/n \in \mathbb Q \cup \{1/0\}$.  We will  
use continued fraction expansion of rational numbers with the following convention:
\[ [b_1,\dots,b_n] = \frac{1}{b_1 - \frac{1}{b_2 - \cdots - \frac{1}{b_n}}} \]
To simplify our discussion we will refer to a knot $K$ that corresponds to $m/n$
as  ``the knot $[b_1,\dots,b_n]$'',  where $m/n = [b_1,\dots,b_n]$.
\end{notation}

Fix integers $r \geq 3$ and $s \geq 2$.  Then for each integer $n$, let $K^n = K^n_{r,s}$ be the knot $[r,-s,n]$.  Most of the discussion in this paper holds
for any such $r, s, n$, except in Section~\ref{sec:small} where we need $n\geq2$ and $n\neq r, r-1$ in order to prove that
$ K^n_{r,s}$ is small and in Section~\ref{sec:hyperbolicity} where we need to assume that $r$, $s$, and $n$ are sufficiently large in order to prove that $K^n_{r,s}$ is hyperbolic.

The knots $K^n$  are related by Dehn twists along a curve  in $T$, say
 $K^{\rm tw} = [r,-s]$, as we now explain.   
 By expanding the continued fractions for $[K^{\rm tw}]$ and $[K^0]$
one sees that $[K^{\rm tw}] = s[a]+(rs+1)[b]$ and $[K^0] = [a]+r[b]$.
The geometric intersection number of these knots is given by the absolute value of the following determinant:
\[\det \left( \begin{matrix}s & rs+1 \\ 1 & r \end{matrix} \right) = sr - (rs+1) = -1\] 
Hence  $[K^{\rm tw}]$ and $[K^0]$ intersect once.
By expanding the continued fraction of $[K^n]$ we see that
$[K^n] = (sn+1)[a]+(rsn+r+n)[b]$; it follows that $[K^n] = n[K^{\rm tw}] + [K^0]$.  
We use this Dehn twist relationship to give an alternative description of \kn\  that will enable us to apply~\cite{BGL} to obtain a lower bound on the torus bridge indices of \kn.

\medskip

First let us recall the realization of Dehn twists along a curve in a surface by Dehn surgery;
see Definition~1.1 of~\cite{BGL}.
Given a curve $c$ in an oriented surface $F$ in a  $3$--manifold $M$, let $c_- \cup c_+$ 
be the link of the negative and positive push-offs of $c$ and $\widehat R$ be the obvious annulus 
bounded by $c_- \cup c_+$ such that $F \cap \widehat{R} = c$.
Then, using their framings as push-offs, $(\tfrac{1}{k}, -\tfrac{1}{k})$--surgery on $c_- \cup c_+$ effects a homeomorphism of $M$ with support in $N(\widehat{R})$ that restricts to the homeomorphism $\tau_c \colon F \to F$ of $k$ Dehn twists along $c$.  This homeomorphism of $M$ is called an {\em annular twist} or a {\em twist along an annulus}. Now, for the situation at hand:

\begin{notation}
\label{notation:HL1L2}
Let $T_-$ and $T_+$ be two distinct fibers
of $E(3_1)$, disjoint from the fiber $T$.
Parametrize the component of $E(3_1)$ 
cut open along $T_- \cup T_+$ that contains $T$
as $T \times [-1,1]$,
with $T_-$ corresponding to $T \times \{-1\}$, $T_+$ 
corresponding to $T \times \{1\}$, and $T$
corresponding to $T \times \{0\}$.  Let $K$ denote a copy of \k0\
on $T$, $\ktm$ a copy of \kt\ on $T_-$, and $\ktp$ a copy of \kt\ on $T_+$.  
Let $\widehat R$ be the annulus that corresponds to
$\kt\ \times [-1,1]$.  Then, because $K^n$ is obtained by Dehn twisting $K^0$ along $\kt$,  \kn\ is 
obtained from \k0\ by twisting $n$ times along $\widehat R$.   
This notation is chosen to be consistent with the notation 
used in~\cite{BGL}.
\end{notation}

The annulus $\widehat R$ and the slopes it defines on $\ktm$, $\ktp$ play a key role in~\cite{BGL}
and we will need three facts about them:

\begin{lem}\label{lem:knotgenus}
The knot $\kt = [r,-s]$ has genus $(s^2(r^2-r+1)-s)/2$.
\end{lem}

\begin{proof}
This lemma is a special case of the theorem that was proved in
Appendix~B of~\cite{baker-diss} and appeared more recently in~\cite{baker-motegi};
we sketch the proof here and we refer the reader to these texts for a detailed discussion.
We view the left image of Figure~\ref{fig:FiberT} as a disk with two handles.
Sliding the left end of the handle for $a$ to the right around the handle for $b$ and back 
changes the crossing of the handles, and we obtain the surface shown in
Figure~\ref{fig:braidedbasis} with a new basis for $H_1(T)$ given by $[b],[c]$ where $[c]=[a]+[b]$.  
As shown in the right part of Figure~\ref{fig:braidedbasis},
for any non-negative coprime integers $p,q$, we get a representation of the curve whose
homology class is $p[c]+q[b]$ as a positive braid.
%
This braid has $p+q$ strands and $pq + p(p-1) + q(q-1)$ crossings.   
Hence $C$ is a fibered knot with a fiber of euler characteristic $p+q - (pq + p^2-p  + q^2-q) = 2(p+q)-(p^2+q^2+pq) $ \cite[Theorem 2]{stallings}.  
Thus $C$ has genus $(p^2+q^2+pq+1)/2 - (p+q)$.
\begin{figure}[h!]
\psfrag{a}[l][l]{$a$}
\psfrag{b}[l][l]{$b$}
\psfrag{c}[l][l]{$c$}
\includegraphics[height=2in]{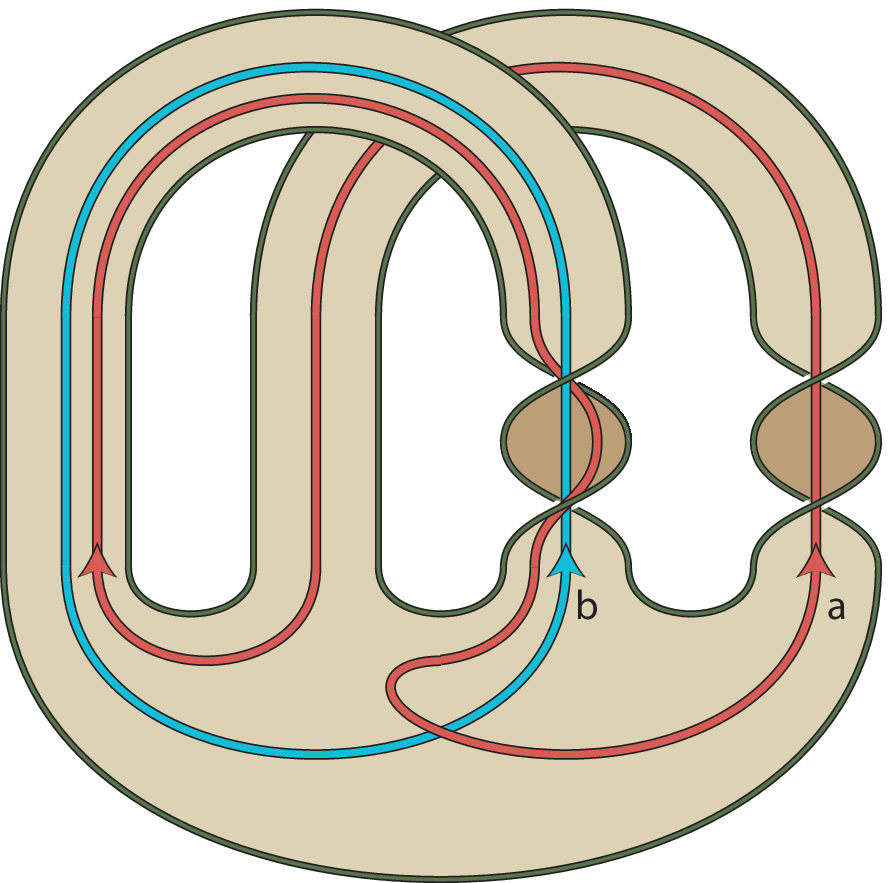} \qquad \qquad \includegraphics[height=2in]{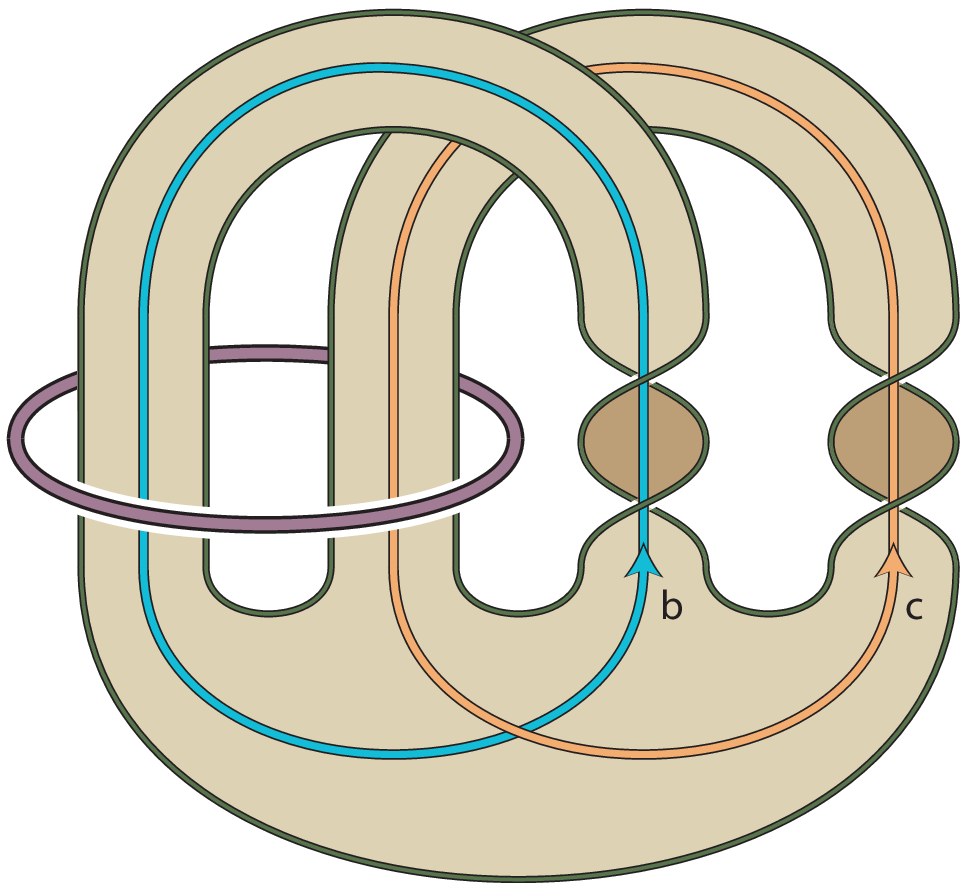}
\caption{Any curve expressed as a non-negative combination of the basis curves $c$ and $b$ is a positive braid with the braid axis shown.  }
\label{fig:braidedbasis}
\end{figure}

Turning to our case at hand, since $[\kt] = s[a] + (rs+1)[b] = s[c]+((r-1)s +1)[b]$ and $r,s>1$, we have that the genus of $\kt$ is $(s^2(r^2-r+1)-s)/2$.
\end{proof}

\begin{lem}
\label{lem:slopesL1L2}
The slope defined on $\ktm$ by $\widehat R$ is not the $0$-slope.
\end{lem}

\begin{proof}
By construction, $\ktm \subset T_-$ and the slope defined on $\ktm$
by $\widehat R$ is the same as the slope defined on $\ktm$ by $T_-$. 
Assume this slope on $\ktm$ is $0$.  Let $K'$ 
be a curve on $T_-$ intersecting $\ktm$ 
once.  The Seifert matrix for $3_{1}$ with respect to Seifert surface $T_-$
and the basis given by \ktm\ and $K'$ is 
$\displaystyle \left(\begin{smallmatrix} 0 & j \\ j\pm1 & k \end{smallmatrix}\right)$, 
where $k$ is the slope defined on $K'$ by $T_-$ and $j$ is the linking number of $K'$ and a 
push-off of $\ktm$.  This implies the leading coefficient of the Alexander polynomial of $3_1$ is $j(j\pm1)$ up to sign.  
But this can never be $\pm1$, a contradiction.  
\end{proof}

\begin{lem}
\label{lem:RnotToral}
The annulus $\widehat R$ does not embed in an unknotted torus in $S^3$.
\end{lem}

\begin{proof}
First,  since $r \ge 3$ and $s \ge 2$, the knot $\kt$ is not an unknot by Lemma~\ref{lem:knotgenus}.  
Hence $\ktm$ is also not an unknot.
Now assume, for a contradiction, that $\widehat R$ does embed in the unknotted torus in $S^3$.
Then, since $\ktm$ is not an unknot, performing surgery on $\ktm$ along the slope defined by $\widehat R$ would yield a
connected sum of two lens spaces, neither of which is $S^3$.
On the other hand, as a framed knot in the fiber of the trefoil, $\ktm$ is a doubly primitive knot 
(this was first proved by Berge in an unpublished manuscript \cite{berge};
see, for example, the Appendix of \cite{saito}).
Thus surgery on \ktm\ with that slope produces a lens space. 
This contradicts uniqueness of prime decomposition of 3-manifolds~\cite{milnor}.
\end{proof}

It is well known that every knot on $3_1$ admits a genus 2 Heegaard surface;
for completeness we prove this here:

\begin{lem}
\label{lem:GenusTwo}
For any $r,s,n$ as above, $g(E(\krsn)) \leq 2$.
\end{lem}

\begin{proof}
Let $T'$, $T_-'$, and $T_+'$ be Seifert surfaces for $3_1$ whose intersection with $E(3_1)$ is 
$T$, $T_-$, and $T_+$ respectively.  Since $E(3_1)$ fibers over the circle, it is easy to see that
$T_-' \cup T_+'$ provides a genus $2$ Heegaard splitting for $S^3$.
Let $H$ be the handlebody containing $T'$; it is natural to identify $H$ with
$T' \times [-1,1]$.  Then $\krsn \subset \mbox{int}T'$ cobounds the annulus  
$A = \krsn \times [-1,0]$ with $\partial H$.  Since $T'$ is a once punctures torus, there exists an arc
$\alpha \subset T'$ intersecting $\krsn$ transversaly once; then $D = \alpha \times [-1,1]$ is
a compressing disk for $\partial H$ intersecting $\krsn$ transversaly once.  Using $A$ and $D$,
it is easy to see that $H \setminus N(\krsn)$ is a compression body.  The lemma follows.
\end{proof}

\section{Smallness}
\label{sec:small}

Fix $r \geq 3$ and $s \geq 2$.  
Our goal in this section is to show that for sufficiently large $n$, \krsn\ is small.  This is a direct
application of Baker's main result in~\cite{baker} (see also \cite{baker-diss}) which we now explain.
Baker analyzes essential closed surfaces in the exterior of the knot
$[b_1,\dots,b_n]$ when $|b_1| \geq 3$ and for $i \geq 2$, 
$|b_i| \geq 2$.  He shows that essential closed surfaces in the exterior of 
$[b_1,\dots,b_n]$ are in one to one correspondence with pairs of sets
$I, J \subset \{1,\dots,n\}$ fulfilling the following three requirements:
\begin{enumerate}
\item $1 \not\in I \cap J$.
\item  Neither $I$ nor $J$ admits consecutive indices.
\item  If $1 \in I$, then $\sum_{i \in J} b_i - \sum_{i \in I} b_i = 0$.
\item  If $1 \not\in I$, then $\sum_{i \in J} b_i - ((\sum_{i \in I} b_i) + 1) = 0$.
\end{enumerate}
We will show that \krsn\ is small by showing that there are no such sets $I,J$.

\begin{pro}
\label{pro:small}
For any $n \geq 2$ and $n \neq r, r-1$, \krsn\ is small.
\end{pro}

\begin{proof}
Recall that \krsn\ is $[r,-s,n]$ and that $r \geq 3$, $s \geq 2$.

By Condition~2 above the possibilities for $I, J \subset \{1,2,3\}$ are $\emptyset,\{1\},\{2\},\{3\},$ and $\{1,3\}$.
The table below describes the sum that appears in the final two equations.  The entries of the table are as follows:
\begin{enumerate}
\item The first row is the set $I$. 
\item The second row records either $- \sum_{i \in I} b_i$ when $1 \in I$
or $- ((\sum_{i \in I} b_i) + 1)$ when $1 \not\in I$.
\item The first column is the set $J$. 
\item The second column records $\sum_{i \in J} b_i$.
\item The remaining entries correspond to $I$, $J$ of the appropriate row and column.
The entry is the sum of the entry in the second row and that in the second column.
Note that these entries are exactly the sum given in Conditions~(3) and ~(4) above,
and must therefore be zero if a closed essential surface is to exist.  The exception are
the four entries that correspond to $I,J$ both being $\{1\}$ or $\{1,3\}$, which are ruled out by
Condition~(1); we entered \X\ there.
\end{enumerate}

\begin{center}
{\renewcommand{\arraystretch}{1.3}
\begin{tabular}{@{}cc|c@{\quad}c@{\quad}c@{\quad}c@{\quad}c@{}}
\toprule
&   &$\emptyset$ & $\{1\}$ & $\{2\}$ & $\{3\}$ & $\{1,3\}$ \\ 
& &-1 & $-r$ & $s-1$ & $-n-1$ &  $-(r+n)$ \\  
\hline
$\emptyset$ &$0$ & -1 & $-r$ & $s-1$ & $-n-1$ &  $-(r+n)$ \\ 
$\{1\}$ & $+r$ & $r-1$ & \X & $r+s-1$ & \textcolor{red}{$r-n-1$} & \X      \\
$\{2\}$ & $-s$ & $-(s+1)$ & $-(r+s)$ & $-1$ & $-(s + n+1)$ & $-(r+s+n)$ \\ 
$\{3\}$ & $+n$ & $n-1$ & \textcolor{red}{$n-r$} & $s+n-1$ & $-1$ & $-r$  \\ 
$\{1,3\}$ & $r+n$ & $r+n-1$ & \X & $r+s+n-1$ & $r-1$ & \X    \\
\bottomrule
\end{tabular}
}
\end{center}

Since  $r\geq 3$, $s \geq 2$, and $n \geq 2$, 
the only entries that could be  
zero are $r-n-1$ and $n-r$.  Our assumption shows that
these are not zero as well.

This completes the proof of Proposition~\ref{pro:small}.
\end{proof}

\section{Bridge index}
\label{sec:bridge}
Fix $r \geq 3$ and $s \geq 2$.  Let \kn\ be the knot \krsn.

Recall that we use the notation $b_1(\kn)$ for the torus bridge index of $\kn$.
In this section we apply the work of Baker--Gordon--Luecke~\cite{BGL} to prove:

\begin{pro}
\label{pro:bridge}
For \kn\ we have:
\[\lim_{n \to \infty} b_1(\kn) = \infty\]
\end{pro}

\begin{proof}
Recall Notation~\ref{notation:HL1L2}.

In~\cite{BGL} the property {\it caught} is defined for a three component link $L_1,L_2,K$
(we remark that the roles played by $L_{1}$ and $L_{2}$ in~\cite{BGL} are different to the role played by $K$).  
By Lemma~2.5 of that paper (numbers refers to version~1 of the arxiv, see bibliography),
if the linking number of $L_1$ and $L_2$ is not zero, then the link $L_1,L_2,K$ is caught.
Applying this in our situation, it is easy to see that the linking number of $\ktm$
and $\ktp$ is zero if and only if the slope defined by $\widehat R$ on $\ktm$ is the zero slope;
by Lemma~\ref{lem:slopesL1L2} this is not the case.
Hence the link $\ktm,\ktp,\k0$ is caught and we may apply Corollary~1.3 of~\cite{BGL}
to conclude that one of the following holds:
	\begin{enumerate}
	\item $\widehat R$ can be isotoped into an unknotted torus in $S^3$; or
	\item  There is an essential annulus $A$ properly embedded in $S^3 \setminus N(L_1 \cup L_2 \cup \k0)$
	with a boundary component in each of $\partial N(L_1)$ and $\partial N(L_2)$ 
	such that $\partial A$ and $\partial \widehat R$ have the same slope on $\partial N(L_1)$ and $\partial N(L_2)$ ; or	
	\item $\lim_{n \to \infty} b_1(\kn) = \infty$.
	\end{enumerate}
We now show that Conclusions~(1) and~(2) do not happen: by Lemma~\ref{lem:RnotToral}, (1) does not happen.
If~(2) happened and $A$ as above existed, then $A \cup \widehat R$ would form a torus immersed in $S^3$, intersecting \k0\
transversally exactly once.  This is impossible as $S^3$ is a homology sphere.

Thus Conclusion~(3) holds, completing the proof of Proposition~\ref{pro:bridge}.
\end{proof}

\section{Hypebolicity}
\label{sec:hyperbolicity}

It remains to show that for sufficiently large $r,s$ and $n$, \krsn\ is
hyperbolic.  Let $\mathcal{C}_7$ be the minimally twisted seven
component link, see the right side of Figure~\ref{fig:surgerydiagram}.  
We first prove that $E(K^n_{r,s})$ may be obtained by Dehn filling $E(\mathcal{C}_7)$.  The generalization of 
this for knots in $T$ is the focus of \cite{BakerBerge}. We overview it here in our specific case.

\begin{pro}
\label{pro:MT7L}
$E(\krsn)$ is homeomorphic to the result of filling $6$ components of 
$\partial E(\mathcal{C}_7)$ along the slopes $r, -s, n, -n-1, s, -r$ 
as shown on the right of Figure~\ref{fig:surgerydiagram}.
\end{pro}

\begin{proof}
The knot $K^n_{r,s} = [r,-s,n]$ may be obtained by surgery on the 
link 
\[\mathcal{L}_7 = L_{-3} \cup L_{-2} \cup L_{-1} \cup L_{0} \cup L_{1} \cup L_2 \cup L_3\]
as depicted on the left side of Figure~\ref{fig:surgerydiagram} (compare with the right side of Figure~\ref{fig:FiberT}); 
this is the content of \cite[Section~3.1]{BakerBerge}.  In short, the idea is that a continued fraction 
of odd length confers an expression of the corresponding knot in $T$ as the image of $a$ under a 
sequence of Dehn twists.  In our present situation, we have
\[K^n_{r,s} = \tau_b^{-r} \circ \tau_a^s \circ \tau_b^{-n} ( a )\]
where $\tau_a$ and $\tau_b$ are positive (right-handed) Dehn twists in $T$ along the curves $a$ and $b$ respectively.  Using the Dehn surgery realization of Dehn twists, 
we obtain a surgery description of $K^n_{r,s}$ by  nesting the pairs of push-offs $L_{-3} \cup L_3$ of $b$, $L_{-2} \cup L_{2}$ of $a$, $L_{-1} \cup L_1$ of $b$, together with $L_0 = a$.  
For these surgeries we use the framings induced by the push-offs.

Next, \cite[Theorem~5.1]{BakerBerge} shows that $E(\mathcal{L}_7)$ and $E(\mathcal{C}_7)$ are homeomorphic
by viewing them as double branched 
covers of the same tangle. A byproduct of the proof is a description of how slopes of $\partial E(\mathcal{L}_7)$ 
correspond to slopes of 
$\partial E(\mathcal{C}_7)$.  This correspondence confers the surgery description on $\mathcal{L}_7$ of the knot $K^n_{r,s}$ to the surgery description on $\mathcal{C}_7$ of the knot $K^*$ 
given on the right side of Figure~\ref{fig:surgerydiagram}.  (The meridian and framing by $T$ for $K^n_{r,s}$ correspond to the longitude and meridian of the knot $K^*$ as a component 
of $\mathcal{C}_7$.)    In particular, $E(K^*)$ is the filling of $E(\mathcal{C}_7)$ stated, and it is homeomorphic to $E(K^n_{r,s})$.
\end{proof}

\begin{figure}[h!]
\centering
\psfrag{K}[r][r]{$K^n$}
\psfrag{A}[r][r]{$-1/s$}
\psfrag{B}[r][r]{$1/s$}
\psfrag{C}[r][r]{$1/n$}
\psfrag{D}[r][r]{$1/r$}
\psfrag{E}[r][r]{$-1/r$}
\psfrag{F}[r][r]{$-1/n$}
\psfrag{G}{{\Large $\mathcal{L}_7$}}
\psfrag{k}[r][r]{$K^*$}
\psfrag{a}[r][r]{$r$}
\psfrag{b}[r][r]{$-s$}
\psfrag{c}[l][l]{$n$}
\psfrag{d}[l][l]{$-n+1$}
\psfrag{e}[r][r]{$s$}
\psfrag{f}[r][r]{$-r$}
\psfrag{g}{{\Large$\mathcal{C}_7$}}
\includegraphics[height=2.25in]{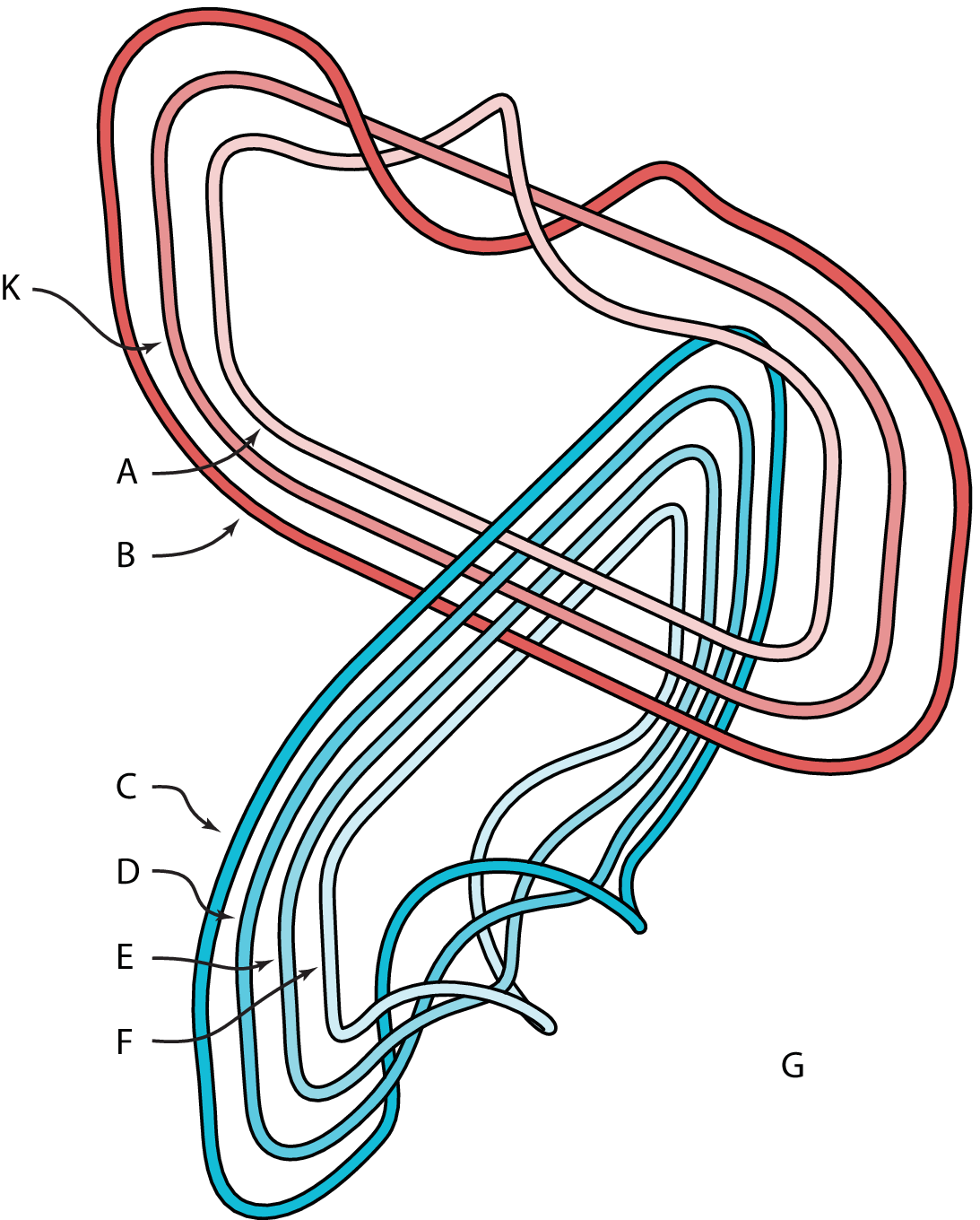} \qquad\qquad \includegraphics[height=2.25in]{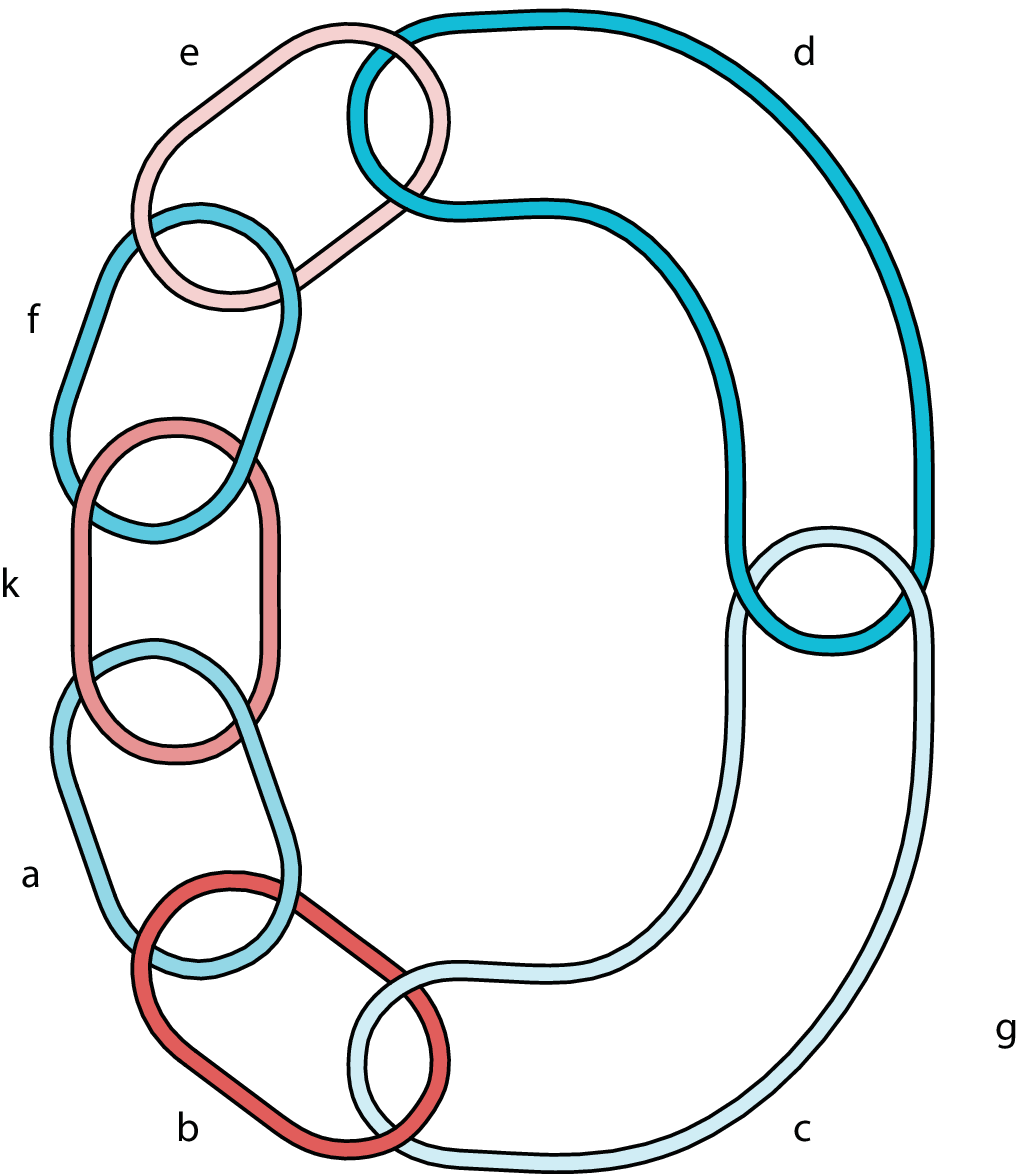}  
\caption{Left: a surgery description of $K^n$ on the link $\mathcal{L}_7$; surgery slopes are taken with respect to framing by fiber. Right: a surgery diagram on $\mathcal{C}_7$ for $K^*$, a knot in a lens space with $E(K^*) \cong E(K^n)$.}
\label{fig:surgerydiagram}
\end{figure}

\begin{pro}
\label{pro:hyperbolicity}
For $r$, $s$, and $n$ sufficiently large, \krsn\ is hyperbolic.
\end{pro}

\begin{proof}
By Neumann and Reid~\cite[Theorem~5.1]{NeumannReid}, $E(\mathcal{C}_7)$ 
is hyperbolic ($\mathcal{C}_7$ is denoted $C(7,-4)$ in~\cite{NeumannReid}).
Therefore, using Proposition~\ref{pro:MT7L}, 
for $r$, $s$, and $n$ sufficiently large,
$E(\krsn)$ is hyperbolic.
\end{proof}

%

\bibliography{BKR}
\bibliographystyle{plain}

\end{document}